\documentclass[dvips,12pt,a4paper]{amsart}

\usepackage{amsmath,amssymb}

\def\M{{\mathcal{M}}}

\def\Z{{\mathbb Z}}
\def\C{{\mathbb C}}

\def\mcO{{\mathcal O}}
\def\bnull{{\vec 0}}
\newcommand{\E}{\overline{\mathcal E}}

\def\ow{{\vec w}}
\def\tT{{\widetilde T}}
\def\P{{\mathbb P}}

\newtheorem{proposition}{Proposition}[section]
\newtheorem{theorem}[proposition]{Theorem}

\newtheorem{lemma}[proposition]{Lemma}
\newtheorem{conjecture}[proposition]{Conjecture}

\newtheorem{remark}[proposition]{Remark}

\title[Homogeneous components in the moduli space of sheaves]{Homogeneous components in the moduli space of sheaves and Virasoro characters}

\subjclass[2010]{14C05, 05A17}

\keywords{Moduli space of sheaves, Virasoro character, quiver variety}

\author{A. Buryak} 
\thanks{A. B. is the corresponding author}

\address{A.~Buryak:\newline
Department of Mathematics,
University of Amsterdam, \newline
P.~O.~Box 94248, 1090 GE Amsterdam, 
The Netherlands\newline 
\indent and\newline
Department of Mathematics, Moscow State University,\newline
Leninskie gory, 119992 GSP-2 Moscow, Russia} 
\email{a.y.buryak@uva.nl, buryaksh@mail.ru}

\author{B. L. Feigin}

\address{B.~L.~Feigin:\newline
Landau Institute for Theoretical Physics, Russia, Chernogolovka, 142432, prosp. Akademika Semenova, 1a,\newline
Higher School of Economics, Russia, Moscow, 101000, Myasnitskaya ul., 20 and \newline
Independent University of Moscow, Russia, Moscow, 119002, Bol'shoi Vlas'evski per., 11}
\email{borfeigin@gmail.com}

\begin{document}

\begin{abstract}
The moduli space $\M(r,n)$ of framed torsion free sheaves on the projective plane with rank $r$ and second Chern class equal to $n$ has the natural action of the $(r+2)$-dimensional torus. In this paper, we look at the fixed point set of different one-dimensional subtori in this torus. We prove that in the homogeneous case the generating series of the numbers of the irreducible components has a beautiful decomposition into an infinite product. In the case of odd $r$ these infinite products coincide with certain Virasoro characters. We also propose a conjecture in a general quasihomogeneous case.
\end{abstract}

\maketitle

\section{Introduction}

Let $\M(r,n)$ be the moduli space of framed torsion free sheaves on $\P^2$ with rank $r$ and second Chern class $c_2$ equal to $n$. It is a smooth irreducible quasi-projective variety of dimension $2rn$. In the case $r=1$ it is isomorphic to the Hilbert scheme of $n$ points on the plane. The moduli space $\M(r,n)$ has a simple quiver description and we recall it in Section \ref{subsection:quiver description}. In principle, one can use this description as a definition of $\M(r,n)$. We refer the reader to \cite{Nakajima} and \cite{Nakajima2} for a more detailed discussion of the moduli space $\M(r,n)$.   

There is a natural action of the $(r+2)$-dimensional torus $T=(\C^*)^{r+2}$ on $\M(r,n)$. It is induced by the $(\C^*)^2$-action on $\P^2$ and by the action of $(\C^*)^r$ on the framing. Consider a vector 
$$
\ow=(w_1,w_2,\ldots,w_r)\in\Z^r
$$
and integers $\alpha,\beta\ge 1$, such that $gcd(\alpha,\beta)=1$. Let $T_{\alpha,\beta}^{\ow}$ be the one-dimensional subtorus of $T$ defined by 
$$
T_{\alpha,\beta}^{\ow}=\{(t^{\alpha},t^{\beta},t^{w_1},t^{w_2},\ldots,t^{w_r})\in T|t\in\C^*\}.
$$
For $0\le m\le r$, let $\ow(m)$ be the vector $(\underbrace{1,\ldots,1}_{\text{$m$ times}},0,\ldots,0)\in\Z^r$. We denote by $h_0(X)$ the number of connected components of a manifold~$X$. We will use the classical $q$-series notations:
\begin{align*}
&(a)_n=(a;q)_n=(1-a)(1-aq)\ldots(1-aq^{n-1}),\\
&(a_1,a_2,\ldots,a_k;q)_{\infty}=(a_1;q)_{\infty}(a_2;q)_{\infty}\ldots(a_k;q)_{\infty}.
\end{align*}
Now we can state our main result.
\begin{theorem}\label{main theorem}
For any $0\le m\le r$ we have
\begin{gather}\label{equation:main theorem}
\sum_{n\ge 0}h_0\left(\M(r,n)^{T_{1,1}^{\ow(m)}}\right)q^n=\frac{(-q)_\infty}{(q)_\infty}(q^{m+1},q^{r-m+1},q^{r+2};q^{r+2})_\infty.
\end{gather}
\end{theorem}

In the case of odd $r$ the right-hand side of \eqref{equation:main theorem} up to the factor $(-q)_{\infty}$ coincides with a certain Virasoro character. We discuss it in Section~\ref{virasoro}. In Section~\ref{conjecture} we formulate a conjecture in the case of arbitrary $\alpha,\beta$. We also give a conjectural formula for the two-variable generating function of the Betti numbers of $\M(2,n)^{T_{1,1}^{\ow(m)}}$ for $m=0,1$.

A connection between the moduli space $\M(r,n)$ and the Virasoro characters (or more generally $W_n$-characters) was also found in \cite{FFJMM}. It appears in a different context and we don't know how to relate it to our work. However in Section \ref{quantum continious} we review briefly the paper \cite{FFJMM}, because we use the characters defined there in our Conjecture~\ref{main conjecture}.   

Our proof of Theorem~\ref{main theorem} is combinatorial but we can propose another way to prove it using the representation theory of the toroidal Yangian. These ideas are under development and we briefly discuss them in Section~\ref{Yangian}. We are going to write the details in the forthcoming paper.

This work is a continuation of \cite{Buryak1} and \cite{Buryak2}. In \cite{Buryak1} the first author studied cohomology groups of $\M(1,n)^{T_{\alpha,\beta}}$. In \cite{Buryak2} the first author computed Betti numbers of $\M(r,n)^{(\C^*)^2}$ and showed that they coincide with certain coefficients in a generalization of the MacMahon's formula.

\subsection{Moduli space of sheaves on $\P^2$} 

The moduli space $\M(r,n)$ is defined by
\begin{gather*}
\M(r,n)=\left.\left\{(E,\Phi)\left|
\begin{smallmatrix}
\text{$E$: a torsion free sheaf on $\P^2$}\\
rank(E)=r,\,c_2(E)=n\\
\text{$\Phi\colon E|_{l_{\infty}}\xrightarrow{\sim}\mcO^{\oplus r}_{l_{\infty}}$: framing at infinity}
\end{smallmatrix}\right.\right\}\right/{\text{isomorphism}},
\end{gather*}
where $l_{\infty}=\{[0:z_1:z_2]\in\P^2\}\subset\P^2$ is the line at infinity. 

Let $\tT$ be the maximal torus of $GL_r(\C)$ consisting of diagonal matrices and let $T=(\C^*)^2\times \tT$. The action of $T$ on $\M(r,n)$ is defined as follows. For $(t_1,t_2)\in(\C^*)^2$ let $F_{t_1,t_2}$ be the automorphism of $\P^2$ defined by
$$
F_{t_1,t_2}([z_0:z_1:z_2])=[z_0:t_1z_1:t_2z_2].
$$
For $diag(e_1,\ldots,e_r)\in \tT$ let $G_{e_1,\ldots,e_r}$ denote the isomorphism of $\mcO^{\oplus r}_{l_{\infty}}$ given by
$$
\mcO^{\oplus r}_{l_\infty}\ni (s_1,\ldots,s_r)\mapsto (e_1s_1,\ldots,e_rs_r).
$$
Then for $(E,\Phi)\in\M(r,n)$ we define
$$
(t_1,t_2,e_1,\ldots,e_r)\cdot (E,\Phi)=((F_{t_1,t_2}^{-1})^*E,\Phi'),
$$
where $\Phi'$ is the composition of the homomorphisms
$$
(F_{t_1,t_2}^{-1})^*E|_{l_{\infty}}\xrightarrow{(F_{t_1,t_2}^{-1})^*\Phi}(F_{t_1,t_2}^{-1})^*\mcO^{\oplus r}_{l_\infty}=\mcO^{\oplus r}_{l_\infty}\xrightarrow{G_{e_1,\ldots,e_r}}\mcO^{\oplus r}_{l_\infty}.
$$

\subsection{Virasoro characters}\label{virasoro}

We recall several results from the representation theory of the Virasoro algebra. There are modules $M(p,p')_2$ that are called the Virasoro minimal models and labelled by coprime integers $p$ and $p'$ for which $1<p<p'$. They contain irreducible modules labelled by $r$ and $s$ with $1\le r<p$ and $1\le s<p'$. In \cite{FF,R}, the characters of these modules were computed to be $\widehat{\overline\chi}^{p,p'}_{r,s}=q^{\Delta^{p,p'}_{r,s}}\overline\chi^{p,p'}_{r,s}$, where $\overline\chi^{p,p'}_{r,s}$ is called the normalized character and is given by:
$$
\overline\chi^{p,p'}_{r,s}=\frac{1}{(q)_\infty}\sum_{\lambda=-\infty}^{\infty}\left(q^{\lambda^2pp'+\lambda(p'r-ps)}-q^{(\lambda p+r)(\lambda p'+s)}\right),
$$
and the number $\Delta^{p,p'}_{r,s}$ is called the conformal dimension and is given by: 
$$
\Delta^{p,p'}_{r,s}=\frac{(p'r-ps)^2-(p'-p)^2}{4pp'}.
$$

Now let us return to Theorem~\ref{main theorem}. The right-hand side of~\eqref{equation:main theorem} is known to be equal to $(-q)_\infty\overline\chi^{2,r+2}_{1,m+1}$ when $r$ is odd (see e.g.\cite{Welsh}). Thus we have the following equation
$$
\sum_{n\ge 0}h_0\left(\M(2k+1,n)^{T_{1,1}^{\ow(m)}}\right)q^n=(-q)_{\infty}\overline\chi^{2,2k+3}_{1,m+1}.
$$

\subsection{Quantum continuous $\mathfrak{gl}_{\infty}$}\label{quantum continious}

In \cite{FFJMM} the authors study representations of the associative algebra which they denote by $\mathcal E$ and call quantum continuous $\mathfrak{gl}_{\infty}$. This algebra depends on parameters $q_1$ and $q_2$. They construct representations $\mathcal F_1(u_1)\otimes\ldots\otimes\mathcal F_s(u_s)$ that depend on parameters $q_1,q_2,u_1,\ldots,u_s$. In \cite{FFJMM} these representations are constructed purely algebraically but they have the following geometrical meaning. The space $\mathcal F_1(u_1)\otimes\ldots\otimes\mathcal F_s(u_s)$ can be identified with the equivariant $K$-theory of $\coprod_{n\ge 0}\M(s,n)$ and the algebra $\mathcal E$ acts there through a slight generalization of the correspodences from~\cite{FT}.  

The authors of \cite{FFJMM} impose the following conditions on the parameters $q_1,q_2,u_1,\ldots,u_s$:
\begin{gather*}
u_i=u_{i+1}q_1^{a_i+1}q_2^{b_i+1},i=1,\ldots,s-1,\qquad q_1^{p}q_2^{p'}=1,
\end{gather*}
where $\vec a=(a_1,\ldots,a_{s-1})\in\Z_{\ge 0}^{s-1}$ and $\vec b=(b_1,\ldots,b_{s-1})\in\Z_{\ge 0}^{s-1}$ are arbitrary vectors and $p,p'\in\Z_{\ge 0}$ are integers such that $p\ne p'$ and 
$$
p-1-\sum_{i=1}^{s-1}(a_i+1)\ge 0,\qquad p'-1-\sum_{i=1}^{s-1}(b_i+1)\ge 0.
$$ 
They construct a new $\mathcal E$-module as a subquotient of $\mathcal F_1(u_1)\otimes\ldots\otimes\mathcal F_s(u_s)$. This module is denoted by $\M^{p,p'}_{\vec a,\vec b}$ and its character is denoted by $\chi^{p,p'}_{\vec a,\vec b}$. These characters are connected with the Virasoro characters in the following way. In \cite{FFJMM} it is proved that if $p'>p>1,gcd(p',p)=1$ and $s=2$, then
$$
\chi^{p,p'}_{a_1,b_1}=\frac{1}{(q)_{\infty}}\overline\chi^{p,p'}_{a_1+1,b_1+1}.
$$

Let us make a remark about the symmetries of the character $\chi^{p,p'}_{\vec a,\vec b}$. For a vector $\vec c\in\Z^{s-1}$ and an integer $m$ we define the vectors $\tau(\vec c,m)=(\tau(\vec c,m)_1,\ldots,\tau(\vec c,m)_{s-1})$ and $\sigma(\vec c,m)=(\sigma(\vec c,m)_1,\ldots,\sigma(\vec c,m)_{s-1})$ as follows: 
$$
\tau(\vec c,m)_i=c_{i+1},\qquad \sigma(\vec c, m)_i=c_{s+1-i},
$$
where $c_s=m-s-\sum_{i=1}^{s-1}c_i$. Then we have (see \cite{FFJMM})
\begin{gather}\label{symmetries}
\chi^{p,p'}_{\vec a,\vec b}=\chi^{p,p'}_{\tau(\vec a,p),\tau(\vec b,p')}=\chi^{p,p'}_{\sigma(\vec a,p),\sigma(\vec b,p')}.
\end{gather}

\subsection{Conjecture $1$: arbitrary $\alpha,\beta$}\label{conjecture}

Consider a vector $\ow\in\Z^r$ and numbers $\alpha,\beta\ge 1$ such that $0\le w_i<\alpha+\beta$ and $gcd(\alpha,\beta)=1$. Let $a_i$ be the number of $j$ such that $w_j=i$, i.e. $a_i=\sharp\{j|w_j=i\}$. The numbers $\alpha$ and $\alpha+\beta$ are coprime, therefore there exists the unique number $\alpha'$ such that $0\le \alpha'<\alpha+\beta$ and $\alpha'\alpha=1\pmod{\alpha+\beta}$. We define the vector $\vec{a'}=(a'_0,a'_1,\ldots,a'_{\alpha+\beta-1})$ as follows 
$$
a'_i=a_{\alpha' i\pmod{\alpha+\beta}}.
$$
We define the vector $\vec{a''}\in\Z^{\alpha+\beta-1}$ as the vector $\vec{a'}$ without the last coordinate. Let $\bnull=(0,0,\ldots,0)\in\Z^{\alpha+\beta-1}$. 

\begin{conjecture}\label{main conjecture}
$$
\sum_{n\ge 0}h_0\left(\M(r,n)^{T^{\ow}_{\alpha,\beta}}\right)q^n=(q^{\alpha+\beta};q^{\alpha+\beta})_\infty\chi^{\alpha+\beta,\alpha+\beta+r}_{\bnull,\vec{a''}}.
$$
\end{conjecture}

\begin{remark}
We used the multiplication by $\alpha^{-1}\pmod{\alpha+\beta}$ in the definition of $\overline a'$. If one uses the multilication by $\beta^{-1}\pmod{\alpha+\beta}$, then the character $\chi$ will be the same. It follows from~\eqref{symmetries} and the fact that $\beta^{-1}=-\alpha^{-1}\pmod{\alpha+\beta}$. 
\end{remark}

\subsection{Conjecture $2$: Betti numbers}

We denote by $P_q(X)$ the Poincare polynomial $\sum_{i\ge 0}\dim H_i(X)q^{\frac{i}{2}}$ of a manifold $X$. 

In \cite{Buryak1} we proposed the following conjecture 
\begin{gather}\label{old conjecture}
\sum_{n\ge 0}P_q\left(\M(1,n)^{T_{\alpha,\beta}}\right)t^n=\prod_{\substack{i\ge 1\\(\alpha+\beta)\nmid i}}\frac{1}{1-t^i}\prod_{i\ge 1}\frac{1}{1-qt^{(\alpha+\beta)i}}.
\end{gather}
We conjecture an analogue of \eqref{old conjecture} for the case $r=2$ and $\alpha=\beta=1$. 
\begin{conjecture}
\begin{align*}
&\sum_{n\ge 0}P_q\left(\M(2,n)^{T^{(0,0)}_{1,1}})\right)t^n=\prod_{4\nmid i}\frac{1}{(1-t^i)(1-q t^i)}\prod_{i\ge 1}\frac{1}{(1-q t^{4i})(1-q^2t^{4i})},\\
&\sum_{n\ge 0}P_q\left(\M(2,n)^{T^{(0,1)}_{1,1}})\right)t^n=\\
&=\prod_{n\ge 1}\frac{(1-t^{4n-2})}{(1-t^{2n-1})^2(1-qt^{4n-2})^2(1-q^2t^{4n-2})(1-qt^{4n})^2}.
\end{align*}
\end{conjecture}

\subsection{The toroidal Yangian} \label{Yangian}

For simplicity in this section we consider the case $\ow=\vec 0$. 

At the moment we can't relate the varieties $\M(r,n)^{T_{\alpha,\beta}}$ to $\widehat{\widehat{gl}}_{1}$-toroidal algebra. However we can see a relation with the toroidal algebra $\widehat{\widehat{gl}}_{\alpha+\beta}$. Let $\Gamma_{\alpha+\beta}$ be the subgroup of $\C^*\times\C^*$ defined by
$$
\Gamma_{\alpha+\beta}=\left\{(\zeta^s,\zeta^{-s})\left|s=0,1,\ldots,\alpha+\beta-1,\zeta=e^{\frac{2\pi i}{\alpha+\beta}}\right.\right\}.
$$
In \cite{Varagnolo} it is proved that the toroidal Yangian acts on the equivariant homology groups $\bigoplus_n H_*^{\C^*\times\C^*}\left(\M(r,n)^{\Gamma_{\alpha+\beta}}\right)$. We want to consider the localized homology groups $\bigoplus_n H_*^{T_{\alpha,\beta}}\left(\M(r,n)^{\Gamma_{\alpha,\beta}}\right)$ and a filtration in them given by a dimension of a support. This filtration is increasing and the lowest level of it has a basis enumerated by the irreducible components of $\M(r,n)^{T_{\alpha,\beta}}$. The toroidal Yangian also has a filtration such that the lowest level is isomorphic to $\widehat{sl}_{\alpha+\beta}$. The filtration on the Yangian induces a filtration on the representation and we suppose that it is exactly the filtration given by a dimension of a support. Then the lowest level of this filtration is the irreducible integrable representation of $\widehat{sl}_{\alpha+\beta}$ of level $r$. Suppose $\alpha=\beta=1$. It is well-known that the characters of the integrable $\widehat{sl}_2$-modules of level $r$ in the principal grading coincide (up to the factor $(-q)_\infty$) with the characters of the irreducible representations of the Virasoro algebra that come from $(2,r+2)$-models. Thus we get~\eqref{equation:main theorem}. Conjecture~\ref{main conjecture} corresponds to the case of general $\alpha,\beta$. There is also a possible way to apply the representation theory of the toroidal Yangian to the proof of the other conjectures in this paper. We hope to develop these ideas in a forthcoming paper.

\subsection{Organization of the paper}

In Section~\ref{section:moduli space} we recall the quiver description of the moduli space~$\M(r,n)$ and find a sufficient condition for the varieties $\M(r,n)^{T_{\alpha,\beta}^\ow}$ to be compact. Compactness of the varieties~$\M(r,n)^{T_{1,1}^{\ow(m)}}$ is important in the proof of Theorem~\ref{main theorem}. In Section~\ref{section:cellular decomposition} we construct a cellular decomposition of $\M(r,n)^{T^{\ow(m)}_{1,1}}$ and obtain a combinatarial formula for the number of the irreducible components. In Section~\ref{section:proof} we analyse this combinatorial formula and give a proof of Theorem~\ref{main theorem}. 

\subsection{Acknowledgments}
The authors are grateful to S. M. Gusein-Zade, M. Finkelberg, S. Shadrin and J. Stokman for useful discussions. 

A. B. is partially supported by a Vidi grant of the Netherlands Organization of Scientific Research, by the grants RFBR-10-01-00678, NSh-4850.2012.1 and the Moebius Contest Foundation for Young Scientists. Research of B. F. is partially supported by RFBR initiative interdisciplinary project grant 09-02-12446-ofi-m, by RFBR-CNRS grant 09-02-93106, RFBR grants 08-01-00720-a, NSh-3472.2008.2 and 07-01-92214-CNRSL-a.

\section{Moduli space of sheaves on $\P^2$}\label{section:moduli space}

Here we recall the quiver description of the moduli space~$\M(r,n)$ and find a sufficient condition for the varieties $\M(r,n)^{T_{\alpha,\beta}^\ow}$ to be compact.

\subsection{Quiver description of $\M(r,n)$}\label{subsection:quiver description}

The variety $\M(r,n)$ has the following quiver description (see e.g.\cite{Nakajima}).
\begin{gather*}
\M(r,n)\cong\left.\left\{(B_1,B_2,i,j)\left|
\begin{smallmatrix}
1) [B_1,B_2]+ij=0\\
2) \text{(stability) There is no subspace} \\
\text{$S\subsetneq\C^n$ such that $B_{\alpha}(S)\subset S$ ($\alpha=1,2$)}\\
\text{and $\mathop{Im}(i)\subset S$} 
\end{smallmatrix}\right.\right\}\right/GL_n(\C),
\end{gather*}
where $B_1,B_2\in End(\C^n), i\in Hom(\C^r,\C^n)$ and $j\in Hom(\C^n,\C^r)$ with the action of $GL_n(\C)$ given by 
\begin{gather*}
g\cdot(B_1,B_2,i,j)=(gB_1g^{-1},gB_2g^{-1},gi,jg^{-1})
\end{gather*} 
for $g\in GL_n(\C)$. 

In terms of Section~\ref{subsection:quiver description} the $T$-action on $\M(r,n)$ is given by (see e.g.\cite{Nakajima2})
\begin{gather*}
(t_1,t_2,e_1,e_2,\ldots,e_r)\cdot[(B_1,B_2,i,j)]=[(t_1B_1,t_2B_2,ie^{-1},t_1t_2ej)].
\end{gather*}

\subsection{Compactness of $\M(r,n)^{T_{\alpha,\beta}^{\ow}}$}\label{compactness}

\begin{proposition}\label{proposition:compactness}
Suppose that $\max\limits_{1\le i\le r}w_i-\min\limits_{1\le i\le r}w_i<\alpha+\beta$, then for any $n$ the variety $\M(r,n)^{T^{\ow}_{\alpha,\beta}}$ is compact.
\end{proposition}
\begin{proof}
By definition, a point $[(B_1,B_2,i,j)]\in\M(r,n)$ is fixed under the action of ${T^{\ow}_{\alpha,\beta}}$ if and only if there exists a homomorphism $\lambda\colon \C^*\to GL_n(\C)$ satisfying the following conditions: 
\begin{align}\label{formula:fixed point}
t^{\alpha}B_1&=\lambda(t)^{-1}B_1\lambda(t),\notag\\
t^{\beta}B_2&=\lambda(t)^{-1}B_2\lambda(t),\\
i\circ diag(t^{w_1},t^{w_2},\ldots,t^{w_r})^{-1}&=\lambda(t)^{-1}i,\notag\\
t^{\alpha+\beta}diag(t^{w_1},t^{w_2},\ldots,t^{w_r})\circ j&=j\lambda(t).\notag
\end{align} 
Suppose that $[(B_1,B_2,i,j)]$ is a fixed point. Then we have the weight decomposition of $\C^n$ with respect to $\lambda(t)$, i.e. $\C^n=\bigoplus_{k\in\Z} V_k$, where $V_k=\{v\in \C^n|\lambda(t)\cdot v=t^kv\}$. We also have the weight decomposition of $\C^r$, i.e. $\C^r=\bigoplus_{k\in\Z} W_k$, where $W_k=\{v\in \C^r|diag(t^{w_1},\ldots,t^{w_r})\cdot v=t^kv\}$. From the conditions \eqref{formula:fixed point} it follows that the only components of $B_1$, $B_2$, $i$ and $j$ that might survive are 
\begin{align}
B_1&\colon V_k\to V_{k-\alpha},\label{eq:condition1}\\
B_2&\colon V_k\to V_{k-\beta},\label{eq:condition2}\\
i&\colon W_k\to V_k,\notag\\
j&\colon V_k\to W_{k-\alpha-\beta}.\label{eq:condition4}
\end{align}
From the stability condition it follows that 
\begin{gather*}
V_k=0, \text{if $k>\max_{1\le i\le r} w_i$}.
\end{gather*}
Then from the condition $\max\limits_{1\le i\le r}w_i-\min\limits_{1\le i\le r}w_i<\alpha+\beta$ and~\eqref{eq:condition4} it follows that $j=0$. 

Consider the variety $\M_0(r,n)$ from \cite{Nakajima2}. It is defined as an affine algebro-geometric quotient 
$$
\M_0(r,n)=\{(B_1,B_2,i,j)|[B_1,B_2]+ij=0\}//GL_n(\C).
$$
It can be viewed as the set of closed orbits in $\{(B_1,B_2,i,j)|[B_1,B_2]+ij=0\}$. There is a morphism $\pi\colon\M(r,n)\to\M_0(r,n)$. It maps a point $[(B_1,B_2,i,j)]\in \M(r,n)$ to the unique closed orbit that is contained in the closure of the orbit of $(B_1,B_2,i,j)$ in $\{(B_1,B_2,i,j)|[B_1,B_2]+ij=0\}$. The variety $\M_0(r,n)$ is affine and the morphism $\pi$ is projective (see e.g.\cite{Nakajima2}).
 
By \cite{Lusztig} the coordinate ring of $\M_0(r,n)$ is generated by the following two types of functions:
\begin{itemize}
\item[a)]
$tr(B_{a_N}B_{a_{N-1}}\cdots B_{a_1}\colon\C^n\to\C^n)$, where $a_i=1$ or $2$.
\item[b)]
$\chi(jB_{a_N}B_{a_{N-1}}\cdots B_{a_1}i)$, where $a_i=1$ or $2$, and $\chi$ is a linear form on $End(\C^r)$.
\end{itemize}
From~\eqref{eq:condition1} and~\eqref{eq:condition2} it follows that the equation 
\begin{gather}\label{eq:null}
\pi^*f|_{\M(r,n)^{T^{\ow}_{\alpha,\beta}}}=0
\end{gather}
holds for any function $f$ of type a). We observed that for any point $[(B_1,B_2,i,j)]\in\M(r,n)^{T^\ow_{\alpha,\beta}}$ we have $j=0$. Hence, \eqref{eq:null} holds for any function $f$ of type b).

We see that the image of $\M(r,n)^{T^\ow_{\alpha,\beta}}$ under the map $\pi$ is a point. Therefore the variety $\M(r,n)^{T^\ow_{\alpha,\beta}}$ is compact.  
\end{proof}

\section{Cellular decomposition of $\M(r,n)^{T^{\ow(m)}_{1,1}}$}\label{section:cellular decomposition}

In this section we construct a cellular decomposition of $\M(r,n)^{T^{\ow(m)}_{1,1}}$ and obtain a combinatorial formula for the number of the irreducible components.

For a partition $\lambda=\lambda_1,\lambda_2,\ldots,\lambda_k, \lambda_1\ge\lambda_2\ge\ldots\ge\lambda_k>0$ let $|\lambda|=\sum_{i=1}^k\lambda_i$ and $l(\lambda)=k$. We denote by $\mathcal P$ the set of all partitions and by $\mathcal D\mathcal P$ the set of partitions with distinct parts.

Let $S(r,m)$ be the set of $r$-tuples $(\lambda^{(1)},\lambda^{(2)},\ldots,\lambda^{(r)})$ of partitions $\lambda^{(i)}\in\mathcal D\mathcal P$ such that $\lambda^{(i)}_1\le l(\lambda^{(i+1)})+\delta_{i,m}$, for $1\le i\le r-1$. Let 
\begin{gather*}
S(r,m)_n=\left\{(\lambda^{(1)},\ldots,\lambda^{(r)})\in S(r,m)\left|\sum_{i=1}^r|\lambda^{(i)}|=n\right.\right\}.
\end{gather*}
\begin{proposition}\label{proposition:fixed points}
\begin{gather*}
h_0\left(\M(r,n)^{T^{\ow(m)}_{1,1}}\right)=\sharp S(r,m)_n.
\end{gather*}
\end{proposition}
\begin{proof}
The set of fixed points of the $T$-action on $\M(r,n)$ is finite and is parametrized by the set of $r$-tuples $D=(D_1,D_2,\ldots,D_r)$ of Young diagrams $D_i$, such that $\sum_{i=1}^r|D_i|=n$ (see~e.g.\cite{Nakajima2}).  

For a Young diagram $Y$ let 
\begin{align*}
&r_l(Y)=|\{(i,j)\in D|j=l\}|,\\
&c_l(Y)=|\{(i,j)\in D|i=l\}|.
\end{align*}
For a point $s=(i,j)\in\Z_{\ge 0}^2$ let
\begin{align*}
&l_Y(s)=r_j(Y)-i-1,\\
&a_Y(s)=c_i(Y)-j-1,
\end{align*}
see Figure \ref{pic1}. Note that $l_Y(s)$ and $a_Y(s)$ are negative if $s\notin Y$.

\begin{figure}[h]
\begin{picture}(70,40)
\put(0,6){
\multiput(0,0)(0,5){7}{\line(1,0){5}}
\multiput(5,0)(0,5){6}{\line(1,0){5}}
\multiput(10,0)(0,5){6}{\line(1,0){5}}
\multiput(15,0)(0,5){4}{\line(1,0){5}}

\multiput(0,0)(5,0){5}{\line(0,1){5}}
\multiput(0,5)(5,0){5}{\line(0,1){5}}
\multiput(0,10)(5,0){5}{\line(0,1){5}}
\multiput(0,15)(5,0){4}{\line(0,1){5}}
\multiput(0,20)(5,0){4}{\line(0,1){5}}
\multiput(0,25)(5,0){2}{\line(0,1){5}}

\put(6.5,6.5){$s$}
\multiput(11,6)(5,0){2}{$\spadesuit$}
\multiput(5.8,10.8)(0,5){3}{$\heartsuit$}

\put(31,11){$a_Y(s)=$ number of $\heartsuit$}
\put(31,16){$l_Y(s)=$ number of $\spadesuit$}
}
\put(8.5,0){$Y$}
\end{picture}
\caption{}
\label{pic1}
\end{figure}

Let $p$ be the fixed point of the $T$-action corresponding to an $r$-tuple $D$. Let $R(T)=\Z[t_1^{\pm 1},t_2^{\pm 1},e_1^{\pm 1},e_2^{\pm 1},\ldots,e_r^{\pm 1}]$ be the representation ring of $T$. Then the weight decomposition of the tangent space $T_p(\M(r,n)$ of the variety $\M(r,n)$ at the point $p$ is given by~(see~e.g.\cite{Nakajima2})
\begin{gather}\label{weight decomposition}
T_p(\M(r,n))=\sum_{i,j=1}^r e_j e_i^{-1}\left(\sum_{s\in D_i}t_1^{-l_{D_j}(s)}t_2^{a_{D_i}(s)+1}+\sum_{s\in D_j}t_1^{l_{D_i}(s)+1}t_2^{-a_{D_j}(s)}\right).
\end{gather}
Consider an integer $\gamma$ and an integer vector $\vec v=(v_1,\ldots,v_r)$ such that 
\begin{gather}\label{general condition}
v_1\gg v_2\gg\ldots\gg v_r\gg\gamma\gg 1.
\end{gather}
It is easy to see that $\M(r,n)^{T}=\M(r,n)^{T^{\vec v}_{1,\gamma}}$. For a fixed point $p\in\M(r,n)^T$ let $C_p=\{z\in\M(r,n)^{T^{\ow(m)}_{1,1}}|\lim_{t\to 0,t\in T^{\vec{v}}_{1,\gamma}}tz=p\}$. By Proposition~\ref{proposition:compactness} the variety $\M(r,n)^{T_{1,1}^{\ow(m)}}$ is compact, hence it has a cellular decomposition with the cells $C_p$ (see \cite{B1,B2}). From~\eqref{general condition} and~\eqref{weight decomposition} it follows that the complex dimension of the cell $C_p$ is equal to
\begin{align*}
&\sum_{i=1}^r|\{s\in D_i|a_{D_i}(s)+1=l_{D_i}(s)\}|+\\
&+\sum_{r\ge i>j\ge 1}|\{s\in D_i|w_j-w_i-l_{D_j}(s)+a_{D_i}(s)+1=0\}|+\\
&+\sum_{r\ge i>j\ge 1}|\{s\in D_j|w_j-w_i+l_{D_i}(s)+1-a_{D_j}(s)=0\}|, 
\end{align*}
where $(w_1,\ldots,w_r)=\ow(m)$. Therefore, the dimension of the cell $C_p$ is equal to $0$ if and only if the following three conditions hold
\begin{align}
&\{s\in D_i|a_{D_i}(s)+1=l_{D_i}(s)\}=\emptyset, \forall 1\le i\le r,\label{null condition1}\\
&\{s\in D_i|w_j-w_i-l_{D_j}(s)+a_{D_i}(s)+1=0\}=\emptyset, \forall r\ge i>j\ge 1, \label{null condition2}\\
&\{s\in D_j|w_j-w_i+l_{D_i}(s)+1-a_{D_j}(s)=0\}=\emptyset, \forall r\ge i>j\ge 1 \label{null condition3}.
\end{align}
It is sufficient to prove that these equations are equivalent to the following system 
\begin{align}
&D_i\in\mathcal D\mathcal P,\label{null condition4}\\
&c_0(D_i)\le r_0(D_{i+1})+\delta_{i,m}, \label{null condition5}
\end{align} 
where $D_i\in\mathcal D\mathcal P$ means that nonzero lengths of columns of a Young diagram $D_i$ are distinct.
 
Suppose that equations~\eqref{null condition1},\eqref{null condition2},\eqref{null condition3} hold. Condition~\eqref{null condition4} easily follows from~\eqref{null condition1}. Suppose that $c_0(D_i)>r_0(D_{i+1})+\delta_{i,m}$. For a point $s=(0,c_0(D_i)-1)$ we have 
\begin{gather}\label{start}
\delta_{i,m}+l_{D_{i+1}}(s)+1-a_{D_i}(s)=\delta_{i,m}+l_{D_{i+1}}(s)+1\ge 0.
\end{gather}
For a point $s=(0,0)$ we have
\begin{gather}\label{end}
\delta_{i,m}+l_{D_{i+1}}(s)+1-a_{D_i}(s)=\delta_{i,m}+r_0(D_{i+1})-c_0(D_i)+1\le 0.
\end{gather}
Note that for two points $s_1=(0,y)$ and $s_2=(0,y+1)$, where $0\le y<c_0(D_{i})-1$, we have
\begin{multline}
\left(l_{D_{i+1}}(s_2)+1-a_{D_i}(s_2)\right)-\left(l_{D_{i+1}}(s_1)+1-a_{D_i}(s_1)\right)=\\
=l_{D_{i+1}}(s_2)-l_{D_{i+1}}(s_1)+1\le 1.\label{difference}
\end{multline}
From \eqref{start}, \eqref{end} and \eqref{difference} it follows that there exists a number $0\le y\le c_0(D_i)-1$ such that for a point $s=(0,y)$ we have
$$
\delta_{i,m}+l_{D_{i+1}}(s)+1-a_{D_i}(s)=0.
$$
This contradicts \eqref{null condition3}. Thus, we have proved \eqref{null condition5}.

Suppose that equations \eqref{null condition4}, \eqref{null condition5} hold. It is easy to see that \eqref{null condition1} follows from \eqref{null condition4}. Let us prove \eqref{null condition3}. Consider a point $s=(x,y)\in D_j$ and let $r\ge i>j\ge 1$. Let $s_1=(x,0)$ and $s_2=(0,0)$, we have
\begin{multline*}
w_j-w_i+l_{D_i}(s)+1-a_{D_j}(s)\stackrel{\text{by \eqref{null condition4}}}{\ge} w_j-w_i+l_{D_i}(s_1)+1-a_{D_j}(s_1)\stackrel{\text{by \eqref{null condition4}}}{\ge}\\
\ge w_j-w_i+l_{D_i}(s_2)+1-a_{D_j}(s_2)=w_j-w_i+r_0(D_i)-c_0(D_j)+1\stackrel{\text{by \eqref{null condition5}}}{>}0.
\end{multline*}
Let us prove \eqref{null condition2}. Suppose $s\in D_i$ and $r\ge i>j\ge 1$. From \eqref{null condition4} and \eqref{null condition5} it follows that
$$
l_{D_j}(s)\le l_{D_i}(s)+w_j-w_i.
$$
Thus
$$
w_j-w_i-l_{D_j}(s)+a_{D_i}(s)+1\ge-l_{D_i}(s)+a_{D_i}(s)+1\stackrel{\text{by \eqref{null condition4}}}{\ge}1.
$$
This completes the proof of the proposition.
\end{proof}

\section{Proof of Theorem \ref{main theorem}}\label{section:proof}

In this section we prove Theorem~\ref{main theorem}. By Proposition~\ref{proposition:fixed points} we have
$$
\sum_{n\ge 0}h_0\left(\M(r,n)^{T_{1,1}^{\ow(m)}}\right)q^n=\sum_{(\lambda^{(1)},\ldots,\lambda^{(r)})\in S(r,m)}q^{\sum\limits_{i=1}^r|\lambda^{(i)}|}.
$$
In Section~\ref{subsection:fermionic} we obtain fermionic expressions for the right-hand side of this equation. The main idea is to transform them to a known fermionic formula for the second infinite product on the right-hand side of~\eqref{equation:main theorem}. In Section \ref{odd case} we use the Gordon's generalization of the Rogers-Ramanujan identities to finish the proof of the theorem in the case when $r$ is odd. The case of even $r$ is covered by an identity from~\cite{Corteel}, we do it in Section \ref{even case}.  

Clearly, $S(r,r)=S(r,0)$. Therefore we have $h_0\left(\M(r,n)^{T_{1,1}^{\ow(r)}}\right)=h_0\left(\M(r,n)^{T_{1,1}^{\ow(0)}}\right)$. It is also obvious that in the case $m=r$ the right-hand side of~\eqref{equation:main theorem} is the same as in the case $m=0$. Thus it is enough to prove the theorem in the case $0\le m\le r-1$.

\subsection{Fermionic expressions for the generating series}\label{subsection:fermionic}

Let $\lambda=\lambda_1,\lambda_2,\ldots,\lambda_s$ be a partition. We will use the standart notation
\begin{gather*}
(q)_\lambda=(q)_{\lambda_1-\lambda_2}\ldots(q)_{\lambda_{s-1}-\lambda_s}(q)_{\lambda_s}.
\end{gather*}
 
\begin{proposition}
Let $0\le m\le r-1$. Then we have 
\begin{gather*}
\sum_{(\lambda^{(1)},\ldots,\lambda^{(r)})\in S(r,m)}q^{\sum\limits_{i=1}^r|\lambda^{(i)}|}=\sum_{\rho_1\ge\ldots\ge\rho_r}\frac{q^{\sum\limits_{i=1}^r\frac{\rho^2_i+\rho_i}{2}}}{(q)_{\rho}}\left(1+\sum\limits_{i=0}^{m-1} q^{\sum\limits_{j=0}^i(\rho_{r-m+j}+1)}\right).
\end{gather*}
\end{proposition}
\begin{proof}
The $q$-binomial coefficients are defined by
$$
\genfrac[]{0pt}{}{M}{N}_q=
\begin{cases}
\frac{(q)_M}{(q)_N(q)_{M-N}},&\text{if $M\ge N\ge 0$},\\
0,&\text{in other cases}.
\end{cases}
$$
We have (see e.g.\cite{Andrews})
\begin{align}
&\genfrac[]{0pt}{}{M}{N}_q=\genfrac[]{0pt}{}{M-1}{N}_q+q^{M-N}\genfrac[]{0pt}{}{M-1}{N-1}_q,\label{formula1}\\
&\sum_{\substack{\lambda\in\mathcal P\\\lambda_1\le M,l(\lambda)\le N}}q^{|\lambda|}=\genfrac[]{0pt}{}{M+N}{N}_q.\label{bounded partitions}
\end{align}
From \eqref{bounded partitions} it follows that
$$
\sum_{\substack{\lambda\in\mathcal D\mathcal P\\l(\lambda)=N,\lambda_1\le M}}q^{|\lambda|}=q^{\frac{N^2+N}{2}}\genfrac[]{0pt}{}{M}{N}_q.
$$
Therefore, we have
\begin{gather*}
\sum_{(\lambda^{(1)},\ldots,\lambda^{(r)})\in S(r,m)}q^{\sum\limits_{i=1}^r|\lambda^{(i)}|}=\sum_{\rho_1,\ldots,\rho_r}q^{\sum\limits_{i=1}^r\frac{\rho_i^2+\rho_i}{2}}\prod_{i=0}^{r-1}\genfrac[]{0pt}{}{\rho_i+\delta_{i,r-m}}{\rho_{i+1}},
\end{gather*}
where we define $\rho_0$ to be equal to $\infty$. Using \eqref{formula1}, we get
\begin{align*}
&\sum_{\rho_1,\ldots,\rho_r}q^{\sum_{i=1}^r\frac{\rho_i^2+\rho_i}{2}}\prod_{i=0}^{r-1}\genfrac[]{0pt}{}{\rho_i+\delta_{i,r-m}}{\rho_{i+1}}=\\
&=\sum_{\rho_1,\ldots,\rho_r}q^{\sum\limits_{i=1}^r\frac{\rho_i^2+\rho_i}{2}}\prod_{i=0}^{r-1}\genfrac[]{0pt}{}{\rho_i}{\rho_{i+1}}+q^{\rho_{r-m}-\rho_{r-m+1}+1}\sum_{\rho_1,\ldots,\rho_r}q^{\sum\limits_{i=1}^r\frac{\rho_i^2+\rho_i}{2}}\prod_{i=0}^{r-1}\genfrac[]{0pt}{}{\rho_i}{\rho_{i+1}-\delta_{i,r-m}}=\\
&=\sum_{\rho_1,\ldots,\rho_r}q^{\sum\limits_{i=1}^r\frac{\rho_i^2+\rho_i}{2}}\prod_{i=0}^{r-1}\genfrac[]{0pt}{}{\rho_i}{\rho_{i+1}}+q^{\rho_{r-m}+1}\sum_{\rho_1,\ldots,\rho_r}q^{\sum\limits_{i=1}^r\frac{\rho_i^2+\rho_i}{2}}\prod_{i=0}^{r-1}\genfrac[]{0pt}{}{\rho_i+\delta_{i,r-m+1}}{\rho_{i+1}}=...=\\
&=\sum_{\rho_1\ge\ldots\ge\rho_r}\frac{q^{\sum\limits_{i=1}^r\frac{\rho^2_i+\rho_i}{2}}}{(q)_{\rho}}\left(1+\sum\limits_{i=0}^{m-1} q^{\sum\limits_{j=0}^i(\rho_{r-m+j}+1)}\right).
\end{align*}
The proposition is proved.
\end{proof}

\begin{proposition}\label{main proposition}
Let $0\le m\le r-1$. Then we have
\begin{align*}
&\sum_{\lambda_1\ge\ldots\ge\lambda_r}\frac{q^{\sum\limits_{i=1}^r\frac{\lambda^2_i+\lambda_i}{2}}}{(q)_{\lambda}}\left(1+\sum\limits_{i=0}^{m-1} q^{\sum\limits_{j=0}^i(\lambda_{r-m+j}+1)}\right)=\\
&=\sum_{\lambda_1\ge\ldots\ge\lambda_r}\frac{q^{\sum\limits_{i=1}^r\frac{\lambda^2_i+\lambda_i}{2}}}{(q)_{\lambda}}\left(1+\sum\limits_{i=0}^{m'-1} q^{\sum\limits_{j=0}^i(\lambda_{r-1-2j}+1)}\right),
\end{align*}
where $m'=\min(m,r-m)$.
\end{proposition}
Before proving this proposition we introduce the following notation. Suppose $P(x_1,\ldots,x_r,q)$ and $Q(x_1,\ldots,x_r,q)$ are polynomials in $x_1,\ldots,x_r$ and $q$. We will write $P\approx Q$ if 
$$
\sum_{\lambda_1\ge\ldots\ge\lambda_r}\frac{q^{\sum\limits_{i=1}^r\frac{\lambda^2_i+\lambda_i}{2}}}{(q)_{\lambda}}\left(P(q^{\lambda_1},\ldots,q^{\lambda_r},q)-Q(q^{\lambda_1},\ldots,q^{\lambda_r},q)\right)=0.
$$

Proposition \ref{main proposition} says that
$$
\sum_{i=0}^{m-1}q^i\prod_{j=0}^ix_{r-m+j}\approx\sum_{i=0}^{\min(m,r-m)-1}q^i\prod_{j=0}^ix_{r-1-2j}.
$$
We will prove a more general statement.
\begin{proposition}
Suppose $0\le s\le l-1$ and $l\le r$. Then
\begin{gather}\label{main transformation}
\left(\sum_{i=0}^{s-1}q^i\prod_{j=0}^ix_{l-s+j}\right)P(x_{\ge l},q)\approx\left(\sum_{i=0}^{\min(s,l-s)-1}q^i\prod_{j=0}^ix_{l-1-2j}\right)P(x_{\ge l},q),
\end{gather}
where $P(x_{\ge l},q)$ is any polynomial that doesn't depend on $x_1,\ldots,x_{l-1}$.
\end{proposition}
\begin{proof}
We adopt the following conventions, $x_{<1}=0$ and $x_{>r}=1$. 
\begin{lemma}\label{simple transformation}
For $1\le s\le r$ we have
$$
x_s(1+qx_{s+1})P(x_{\ne s},q)\approx x_{s+1}(1+qx_{s-1})P(x_{\ne s},q),
$$
where $P(x_{\ne s},q)$ is a polynomial that doesn't depend on $x_s$.
\end{lemma}
\begin{proof}
We have
\begin{align}
&\sum_{\lambda_1\ge\ldots\ge\lambda_r}\frac{q^{\sum\limits_{i=1}^r\frac{\lambda^2_i+\lambda_i}{2}}}{(q)_{\lambda}}\left(q^{\lambda_s}-q^{\lambda_{s+1}}+q^{\lambda_s+\lambda_{s+1}+1}-q^{\lambda_{s-1}+\lambda_{s+1}+1}\right)P(q^{\lambda_{\ne s}},q)=\notag\\
&=-\sum_{\lambda_1\ge\ldots\ge\lambda_r}\frac{q^{\sum\limits_{i=1}^r\frac{\lambda^2_i+\lambda_i}{2}}}{(q)_{\lambda}}q^{\lambda_{s+1}}(1-q^{\lambda_s-\lambda_{s+1}})P(q^{\lambda_{\ne s}},q)+\label{sum1}\\
&+\sum_{\lambda_1\ge\ldots\ge\lambda_r}\frac{q^{\sum\limits_{i=1}^r\frac{\lambda^2_i+\lambda_i}{2}}}{(q)_{\lambda}}q^{\lambda_s+\lambda_{s+1}+1}(1-q^{\lambda_{s-1}-\lambda_s})P(q^{\lambda_{\ne s}},q)\label{sum2}.
\end{align}
In~\eqref{sum1} we make the shift $\lambda_i\mapsto\lambda_i+1$ for $i=1,\ldots,s$ and in the sum~\eqref{sum2} we make the shift $\lambda_i\mapsto\lambda_i+1$ for $i=1,\ldots,s-1$. We get
\begin{align*}
&-\sum_{\lambda_1\ge\ldots\ge\lambda_r}\frac{q^{\sum\limits_{i=1}^r\frac{\lambda^2_i+\lambda_i}{2}}}{(q)_{\lambda}}q^{\lambda_{s+1}}q^{\sum\limits_{i=1}^s(\lambda_i+1)}P(q^{\lambda_1+1},\ldots,q^{\lambda_{s-1}+1},q^{\lambda_{\ge s+1}},q)+\\
&+\sum_{\lambda_1\ge\ldots\ge\lambda_r}\frac{q^{\sum\limits_{i=1}^r\frac{\lambda^2_i+\lambda_i}{2}}}{(q)_{\lambda}}q^{\lambda_s+\lambda_{s+1}+1}q^{\sum\limits_{i=1}^{s-1}(\lambda_i+1)}P(q^{\lambda_1+1},\ldots,q^{\lambda_{s-1}+1},q^{\lambda_{\ge s+1}},q)=0.
\end{align*}
The lemma is proved.
\end{proof}
\begin{lemma}\label{second transformation}
Suppose that $l\le r$ and $0\le s\le\frac{l-2}{2}$, then
$$
(1+qx_l)\prod_{i=0}^{s}x_{l-1-2i}P(x_{\ge l},q)\approx (1+qx_{l-2s-2})\prod_{i=0}^sx_{l-2i}P(x_{\ge l},q).
$$
\end{lemma}
\begin{proof}
By Lemma \ref{simple transformation}
\begin{align*}
&(1+qx_l)x_{l-1}x_{l-3}\ldots x_{l-1-2s}P(x_{\ge l},q)\approx\\
&\approx x_{l}(1+qx_{l-2})x_{l-3}x_{l-5}\ldots x_{l-1-2s}P(x_{\ge l},q)\approx\\
&\approx x_{l}x_{l-2}(1+qx_{l-4})x_{l-5}\ldots x_{l-1-2s}P(x_{\ge l},q)\approx\ldots \\
&\approx x_{l}x_{l-2}\ldots x_{l-2s}(1+qx_{l-2-2s})P(x_{\ge l},q).
\end{align*}
The lemma is proved.
\end{proof}

We will prove \eqref{main transformation} by induction on $s$. The case $s=1$ is trivial and the case $s=2$ follows from Lemma~\ref{simple transformation}. Suppose $s\ge 3$. We have
\begin{align*}
&\left(\sum_{i=0}^{s-1}q^i\prod_{j=0}^ix_{l-s+j}\right)P(x_{\ge l},q)=\\
&=\left((1+qx_{l-1})\sum_{i=0}^{s-2}q^i\prod_{j=0}^ix_{l-s+j}-qx_{l-1}\sum_{i=0}^{s-3}q^i\prod_{j=0}^ix_{l-s+j}\right)P(x_{\ge l},q).
\end{align*}
Suppose that $2s\le l+1$, then 
\begin{align*}
&\left((1+qx_{l-1})\sum_{i=0}^{s-2}q^i\prod_{j=0}^ix_{l-s+j}-qx_{l-1}\sum_{i=0}^{s-3}q^i\prod_{j=0}^ix_{l-s+j}\right)P(x_{\ge l},q)\stackrel{\substack{\text{by the induction}\\\text{assumption}}}{\approx}\\
&\approx\left((1+qx_{l-1})\sum_{i=0}^{s-2}q^i\prod_{j=0}^ix_{l-2-2j}-\sum_{i=1}^{s-2}q^i\prod_{j=0}^ix_{l-1-2j}\right)P(x_{\ge l},q)\stackrel{\text{by Lemma \ref{second transformation}}}{\approx}\\
&\approx\left(\sum_{i=0}^{s-2}(1+qx_{l-3-2i})q^i\prod_{j=0}^ix_{l-1-2j}-\sum_{i=1}^{s-2}q^i\prod_{j=0}^ix_{l-1-2j}\right)P(x_{\ge l},q)=\\
&=\left(\sum_{i=0}^{s-1}q^i\prod_{j=0}^ix_{l-1-2j}\right)P(x_{\ge l},q).
\end{align*}
We see that we have done the induction step in the case $2s\le l$. If $2s=l+1$, then it remains to note that 
$$
\left(\sum_{i=0}^{s-1}q^i\prod_{j=0}^ix_{2s-2-2j}\right)P(x_{\ge 2s-1},q)=\left(\sum_{i=0}^{s-2}q^i\prod_{j=0}^ix_{2s-2-2j}\right)P(x_{\ge 2s-1},q).
$$
Suppose that $2s\ge l+2$, then
\begin{align*}
&\left((1+qx_{l-1})\sum_{i=0}^{s-2}q^i\prod_{j=0}^ix_{l-s+j}-qx_{l-1}\sum_{i=0}^{s-3}q^i\prod_{j=0}^ix_{l-s+j}\right)P(x_{\ge l},q)\stackrel{\substack{\text{by the induction}\\\text{assumption}}}{\approx}\\
&\approx\left((1+qx_{l-1})\sum_{i=0}^{l-s-1}q^i\prod_{j=0}^ix_{l-2-2j}-\sum_{i=1}^{l-s}q^i\prod_{j=0}^ix_{l-1-2j}\right)P(x_{\ge l},q)\stackrel{\text{by Lemma \ref{second transformation}}}{\approx}\\
&\approx\left(\sum_{i=0}^{l-s-1}(1+qx_{l-3-2i})q^i\prod_{j=0}^ix_{l-1-2j}-\sum_{i=1}^{l-s}q^i\prod_{j=0}^ix_{l-1-2j}\right)P(x_{\ge l},q)=\\
&=\left(\sum_{i=0}^{l-s-1}q^i\prod_{j=0}^ix_{l-1-2j}\right)P(x_{\ge l},q).
\end{align*}
The proposition is proved.
\end{proof}

From Proposition \ref{main proposition} it follows that
$$
h_0\left(\M(r,n)^{T^{\ow(m)}_{1,1}}\right)=h_0\left(\M(r,n)^{T^{\ow(r-m)}_{1,1}}\right).
$$
We can also see that the substitution $m\mapsto r-m$ doesn't change the right-hand side of~\eqref{equation:main theorem}. So in the rest of the proof of the theorem we assume that $m\le \frac{r}{2}$.
 
\subsection{The case $r=2k+1$}\label{odd case}

We have $0\le m\le k$.
\begin{proposition}\label{prop1}
\begin{align}
&\sum_{\lambda_1\ge\ldots\ge\lambda_r}\frac{q^{\sum\limits_{i=1}^r\frac{\lambda^2_i+\lambda_i}{2}}}{(q)_{\lambda}}\left(1+\sum\limits_{i=0}^{m-1} q^{\sum\limits_{j=0}^i(\lambda_{r-1-2j}+1)}\right)=\label{big sum}\\
&=(-q)_{\infty}\sum_{\lambda_1\ge\ldots\ge\lambda_k}\frac{q^{\sum\limits_{i=1}^k(\lambda^2_i+\lambda_i)}}{(q)_{\lambda}}\left(1+\sum\limits_{i=0}^{m-1} q^{\sum\limits_{j=0}^i(\lambda_{k-j}+1)}\right).\label{small sum}
\end{align}
\end{proposition}
\begin{proof}
We have the following equation
\begin{gather}\label{qbin}
\sum_{i=\beta}^{\alpha}\frac{q^{\frac{i(i+1)}{2}}q^{\frac{\beta(\beta+1)}{2}}}{(q)_{\alpha-i}(q)_{i-\beta}}=q^{\beta(\beta+1)}\frac{(-q^{\beta+1})_{\alpha-\beta}}{(q)_{\alpha-\beta}}.
\end{gather}
It can be easily derived from the $q$-binomial formula (see e.g.\cite{Andrews}). 

We can fix $\lambda_2,\lambda_4,\ldots,\lambda_{2k}$ in~\eqref{big sum} and sum over $\lambda_1,\lambda_3,\ldots,\lambda_{2k+1}$ using~\eqref{qbin}. Then we get exactly the sum~\eqref{small sum}. 
\end{proof}

\begin{proposition}\label{prop2}
$$
\sum_{\lambda_1\ge\ldots\ge\lambda_k}\frac{q^{\sum\limits_{i=1}^k(\lambda^2_i+\lambda_i)}}{(q)_{\lambda}}\left(1+\sum\limits_{i=0}^{m-1} q^{\sum\limits_{j=0}^i(\lambda_{k-j}+1)}\right)=\frac{(q^{m+1},q^{2k-m+2},q^{2k+3};q^{2k+3})_{\infty}}{(q)_{\infty}}.
$$
\end{proposition}
\begin{proof}
Consider the functions $J_{k,i}(a,x,q)$ from Ch.7 of the book \cite{Andrews}. We only need the following two properties of $J_{k,i}$ (see \cite[Ch.7]{Andrews}).
\begin{align}
&J_{k,i}(0,x,q)=\sum_{\lambda_1\ge\ldots\ge\lambda_{k-1}}x^{|\lambda|}\frac{q^{\lambda_1^2+\ldots+\lambda_{k-1}^2+\lambda_i+\ldots+\lambda_{k-1}}}{(q)_{\lambda}},\notag\\
&J_{k,i}(0,x,q)-J_{k,i-1}(0,x,q)=(xq)^{i-1}J_{k,k-i+1}(0,xq,q).\label{formula2}
\end{align}
We have
\begin{align*}
&J_{k+1,m+1}(0,1,q)=J_{k+1,m}(0,1,q)+q^mJ_{k+1,k-m+1}(0,q,q)=\\
&=J_{k+1,m-1}(0,1,q)+q^{m-1}J_{k+1,k-m+2}(0,q,q)+q^mJ_{k+1,k-m+1}(0,q,q)=\ldots\\
&=J_{k+1,1}(0,1,q)+\sum_{i=0}^{m-1} q^{i+1}J_{k+1,k-i}(0,q,q)=\\
&=\sum_{\lambda_1\ge\ldots\ge\lambda_k}\frac{q^{\sum\limits_{i=1}^k(\lambda^2_i+\lambda_i)}}{(q)_{\lambda}}\left(1+\sum\limits_{i=0}^{m-1} q^{\sum\limits_{j=0}^i(\lambda_{k-j}+1)}\right).
\end{align*}
On the other hand we have (see \cite[Ch.7]{Andrews})
$$
J_{k+1,m+1}(0,1,q)=\frac{(q^{m+1},q^{2k-m+2},q^{2k+3};q^{2k+3})_{\infty}}{(q)_{\infty}}.
$$
This completes the proof of the proposition.
\end{proof}
Propositions \ref{prop1} and \ref{prop2} conclude the proof of the theorem in the case when $r$ is odd.

\subsection{The case $r=2k$}\label{even case}

We have $0\le m\le k$.
\begin{proposition}\label{proposition4}
\begin{align*}
&\sum_{\lambda_1\ge\ldots\ge\lambda_r}\frac{q^{\sum\limits_{i=1}^r\frac{\lambda^2_i+\lambda_i}{2}}}{(q)_{\lambda}}\left(1+\sum\limits_{i=0}^{m-1} q^{\sum\limits_{j=0}^i(\lambda_{r-1-2j}+1)}\right)=\\
&=\sum_{\lambda_1\ge\ldots\ge\lambda_k}\frac{(-q)_{\lambda_1}q^{\frac{\lambda_1^2+\lambda_1}{2}+\sum\limits_{i=2}^k(\lambda^2_i+\lambda_i)}}{(q)_{\lambda}}\left(1+\sum\limits_{i=0}^{m-1} q^{\sum\limits_{j=0}^i(\lambda_{k-j}+1)}\right).
\end{align*}
\end{proposition}
\begin{proof}
Similar to the proof of Proposition~\ref{prop1}.
\end{proof}

\begin{proposition}\label{second lemma}
We have
\begin{align*}
&\sum_{\lambda_1\ge\ldots\ge\lambda_k}\frac{(-q)_{\lambda_1}q^{\frac{\lambda_1^2+\lambda_1}{2}+\sum\limits_{i=2}^k(\lambda^2_i+\lambda_i)}}{(q)_{\lambda}}\left(2\sum\limits_{i=0}^{m-1} q^{\sum\limits_{j=0}^{i-1}(\lambda_{k-j}+1)}+q^{\sum\limits_{j=0}^{m-1}(\lambda_{k-j}+1)}\right)=\\
&=\sum_{\lambda_1\ge\ldots\ge\lambda_k}\frac{(-q)_{\lambda_1}q^{\frac{\lambda_1^2-\lambda_1}{2}+\sum\limits_{i=2}^{k}\lambda_i^2+\sum\limits_{i=m+1}^k\lambda_i}}{(q)_{\lambda}}.
\end{align*}
\end{proposition}
\begin{proof}

Suppose $P(x_1,\ldots,x_k,q)$ and $Q(x_1,\ldots,x_k,q)$ are polynomials in variables $x_1,\ldots,x_k$ and $q$. We will write $P\approx_2 Q$ if 
$$
\sum_{\lambda_1\ge\ldots\ge\lambda_k}\frac{(-q)_{\lambda_1}q^{\frac{\lambda_1^2-\lambda_1}{2}+\sum\limits_{i=2}^k\lambda^2_i}}{(q)_{\lambda}}\left(P(q^{\lambda_1},\ldots,q^{\lambda_k},q)-Q(q^{\lambda_1},\ldots,q^{\lambda_k},q)\right)=0.
$$
The proposition says that 
\begin{gather}\label{transformation3}
\prod_{i=m+1}^kx_i\approx_2\left(2\sum_{i=0}^{m-1}q^{i}\prod_{j=k-i+1}^kx_j+q^m\prod_{i=k-m+1}^kx_i\right)\prod_{i=1}^kx_i.
\end{gather}
The proof of \eqref{transformation3} is based on the following three lemmas. We adopt the conventions, $x_{<1}=0$ and $x_{>k}=1$.

\begin{lemma}\label{simple transformation2}
Suppose $1\le s\le k$, then we have
\begin{align*}
&x_s(1+qx_sx_{s+1})P(x_{\ne s},q)\approx_2 x_{s+1}(1+qx_{s-1}x_s)P(x_{\ne s},q),&\text{if $s\ge 2$,}\\
&x_1(1+x_2+qx_1x_2)P(x_{\ge 2},q)\approx_2x_2P(x_{\ge 2},q),&\text{if $s=1$.}
\end{align*}
\end{lemma}
\begin{proof}
Similar to the proof of Lemma \ref{simple transformation}.
\end{proof}

\begin{lemma}\label{transformation2}
Suppose $1\le s<l\le k+1$, then
\begin{align}
&(x_l-x_s)\prod_{i=s+1}^{l}x_i\prod_{i=l+1}^{k}x_i^2\approx_2q(x_{l-1}-x_{s-1})\prod_{i=s}^{l-1}x_i\prod_{i=l}^{k}x_i^2,&\text{if $s\ge 2$,}\label{f1}\\
&(x_l-x_1)\prod_{i=2}^lx_i\prod_{i=l+1}^kx_i^2\approx_2(1+qx_{l-1})\prod_{i=1}^{l-1}x_i\prod_{i=l}^kx_i^2,&\text{if $s=1$.}\label{f2}
\end{align}
\end{lemma}
\begin{proof}
Let us prove \eqref{f1}. Using Lemma \ref{simple transformation2} we have
\begin{align*}
&(1+qx_{s-1}x_s)x_{s+1}\prod_{i=s+2}^{l-1}x_i\prod_{i=l}^{k}x_i^2\approx_2\\
&\approx_2 x_s(1+qx_sx_{s+1})x_{s+2}\prod_{i=s+3}^{l-1}x_i\prod_{i=l}^{k}x_i^2\approx_2\ldots\\
&\approx_2 x_sx_{s+1}\ldots x_{l-1}(1+qx_{l-1}x_l)x_l\prod_{i=l+1}^{k}x_i^2.
\end{align*}
Thus, \eqref{f1} is proved. Equation \eqref{f2} can be proved similarly.
\end{proof}

\begin{lemma}\label{trans3}
For any $1\le s\le k$ we have
$$
(1-x_s)\prod_{i=s+1}^kx_i\approx_2q^{s-1}(1+qx_{k-s+1})\prod_{i=1}^{k-s+1}x_i\prod_{i=k-s+2}^kx_i^2.
$$
\end{lemma}

\begin{proof}
By Lemma \ref{transformation2} we have
\begin{align*}
&(1-x_s)\prod_{i=s+1}^kx_i\approx_2q(x_k-x_{s-1})\prod_{i=s}^kx_i\approx_2\\
&\approx_2 q^2(x_{k-1}-x_{s-2})\left(\prod_{i=s-1}^{k-1}x_i\right)x_k^2\approx_2\ldots\\
&\approx_2q^{s-1}(x_{k-s+2}-x_1)\prod_{i=2}^{k-s+2}x_i\prod_{i=k-s+3}^kx_i^2\approx_2\\
&\approx_2q^{s-1}(1+qx_{k-s+1})\prod_{i=1}^{k-s+1}x_i\prod_{i=k-s+2}^kx_i^2.
\end{align*}
The lemma is proved.
\end{proof}

We are ready to prove \eqref{transformation3}. We have
\begin{align*}
&\prod_{i=m+1}^kx_i=\sum_{i=0}^{m-1}(1-x_{m-i})\prod_{i=m-i+1}^kx_i+\prod_{i=1}^kx_i\stackrel{\text{by Lemma \ref{trans3}}}{\approx_2}\\
&\approx_2\left(\sum_{i=0}^{m-1}q^i(1+qx_{k-i})\prod_{j=k-i+1}^kx_j+1\right)\prod_{i=1}^kx_i=\\
&=\left(2\sum_{i=0}^{m-1}q^{i}\prod_{j=k-i+1}^kx_j+q^m\prod_{i=k-m+1}^kx_i\right)\prod_{i=1}^kx_i.
\end{align*}
This completes the proof of the proposition.
\end{proof}

\begin{proposition}\label{proposition3}
\begin{align*}
&\sum_{\lambda_1\ge\ldots\ge\lambda_k}\frac{(-q)_{\lambda_1}q^{\frac{\lambda_1^2+\lambda_1}{2}+\sum\limits_{i=2}^k(\lambda^2_i+\lambda_i)}}{(q)_{\lambda}}\left(1+\sum\limits_{i=0}^{m-1} q^{\sum\limits_{j=0}^i(\lambda_{k-j}+1)}\right)=\\
&=\frac{(-q)_{\infty}}{(q)_{\infty}}(q^{m+1},q^{2k-m+1},q^{2k+2};q^{2k+2})_{\infty}.
\end{align*}
\end{proposition}

\begin{proof}
Consider the functions $\E_{k+1,i}(a,q)$ from the paper \cite{Corteel}. It is proved there that
\begin{align}
&\E_{k+1,m+1}(a,q)=\sum_{\lambda_1\ge\ldots\ge\lambda_k}\frac{q^{\frac{\lambda_1^2+\lambda_1}{2}+\sum\limits_{i=2}^k\lambda_i^2+\sum\limits_{i=m+1}^k\lambda_i}\left(-\frac{1}{a}\right)_{\lambda_1}a^{\lambda_1}}{(q)_{\lambda}},\label{first}\\
&\E_{k+1,m+1}\left(\frac{1}{q},q\right)=\notag\\
&=\frac{(-q)_{\infty}}{(q)_{\infty}}\left((q^{m+1},q^{2k-m+1},q^{2k+2};q^{2k+2})_{\infty}+(q^{m},q^{2k-m+2},q^{2k+2};q^{2k+2})_{\infty}\right).\label{second}
\end{align}
Combining \eqref{first},\eqref{second} and Proposition~\ref{second lemma} we get the proof of the proposition.   
\end{proof}

Propositions \ref{proposition4} and \ref{proposition3} complete the proof of the theorem in the case of even $r$.

\end{document}